\documentclass[anon,12pt]{msml2021} 

\title[Fourier-domain Variational Formulation for Supervised Learning]{Fourier-domain Variational Formulation and Its Well-posedness for Supervised Learning}
 
\usepackage{times}

\usepackage{caption}
\usepackage{mlmath} 
\newtheorem{thm}{Theorem}

\newtheorem{lem}{Lemma}
\newtheorem{prop}{Proposition}
\newtheorem{rmk}{Remark}

\newtheorem{exam}{Example}

\newtheorem{prob}{Problem}

\author{%
	\Name{Tao Luo}\thanks{Corresponding author} \Email{luotao41@sjtu.edu.cn}\\
	\Name{Zheng Ma} \Email{zhengma@sjtu.edu.cn}\\
	\Name{Zhiwei Wang} \Email{victorywzw@sjtu.edu.cn}\\
	\Name{Zhi-Qin John Xu} \Email{xuzhiqin@sjtu.edu.cn}\\
	\Name{Yaoyu Zhang} \Email{zhyy.sjtu@sjtu.edu.cn}\\
	\addr School of Mathematical Sciences, Institute of Natural Sciences, MOE-LSC and Qing Yuan Research Institute, \\
	Shanghai Jiao Tong University, Shanghai, 200240, P.R. China%
}

\makeatletter
 \let\Ginclude@graphics\@org@Ginclude@graphics
\makeatother

\begin{document}

\maketitle

\begin{abstract}%
    A supervised learning problem is to find a function in a hypothesis function space given values on isolated data points. Inspired by the frequency principle in neural networks, we propose a Fourier-domain variational formulation for supervised learning problem. This formulation circumvents the difficulty of imposing the constraints of given values on isolated data points in continuum modelling. Under a necessary and sufficient condition within our unified framework, we establish the well-posedness of the Fourier-domain variational problem, by showing a critical exponent depending on the data dimension. In practice, a neural network can be a convenient way to implement our formulation, which automatically satisfies the well-posedness condition.  
\end{abstract}


\section{Introduction\label{sec:Introduction}}  


Supervised learning is ubiquitous. In a supervised learning problem, the goal is to find a function in a hypothesis function space given values on isolated data points with labels. In practice, Deep neural network (DNN), although with limit understanding, has been a powerful method. A series of works provide a good explanation for the good generalization of DNNs by showing a \emph{Frequency Principle (F-Principle), i.e., a DNN tends to learn a target function from low to high frequencies during the training}~\citep{xu_training_2018,xu2019frequency,rahaman2018spectral}.  The F-Principle shows a low-frequency bias of DNNs when fitting a given data set. In the neural tangent kernel regime \citep{jacot2018neural,lee2019wide}, later works show that the long-time training solution of a wide two-layer neural network is equivalent to the solution of a constrained Fourier-domain variational problem~\citep{zhang_explicitizing_2019,Luotao2020LFP}.

Inspired by above works about the F-Principle, in this paper, we propose a general Fourier-domain variational formulation for supervised learning problem and study its well-posedness. In continuum modelling, it is often difficult to impose the constraint of given values on isolated data points in a function space without sufficient regularity, e.g., a $L^p$ space. We circumvent this difficulty by regarding the Fourier-domain variation as the primal problem and the constraint of isolated data points is imposed through a linear operator. Under a necessary and sufficient condition within our unified framework, we establish the well-posedness of the Fourier-domain variational problem. We show that the well-posedness depends on a critical exponent, which equals to the data dimension. This is a stark difference compared with a traditional partial differential equation (PDE) problem. For example, in a boundary value problem of any PDE in a $d$-dimensional domain, the boundary data should be prescribed on the $(d-1)$-dimensional boundary of the domain, where the dimension $d$ plays an important role. However, in a well-posed supervised learning problem, the constraint is always on isolated points, which are $0$-dimensional independent of $d$, while the model has to satisfy a well-posedness condition depending on the dimension. In practice, a neural network can be a convenient way to implement our formulation, which automatically satisfies the well-posedness condition. With a clear understanding of its posedness, the Fourier-domain variational formulation also provides insight for designing methods for supervised learning problems. 

The rest of the paper is organized as follows. Section~\ref{section2} shows some related work. In section~\ref{section3}, we propose a Fourier-domain variational formulation for supervised learning problems. The necessary and sufficient condition for the well-posedness of our model is presented in section~\ref{section4}. Section~\ref{section5} is devoted to the numerical demonstration in which we solve the Fourier-domain variational problem using band-limited functions. Finally, we present a short conclusion and discussion in section~\ref{sec:conclusion}.
\section{Related Works}\label{section2}

Our work, as a modelling for supervised learning, is related to the point cloud interpolation problem which belongs to semi-supervised learning. One of the most widely used methods for the point cloud problems is the 2-Laplacian method~\citep{zhu2003semi}, which is an approach based on a Gaussian random field and weighted graph model. But it has been observed~\citep{el2016asymptotic,Nadler2009SemisupervisedLW} that when the number of unlabeled data point is large, the graph Laplacian method is usually ill-posed. A new weighted Laplace method was proposed to overcome this shortcoming of the original 2-Laplacian method~\citep{2017Weighted}. 
\cite{2019Properly} further considered a way to correctly set the weights in Laplacian regularization with a exponent $\alpha>d$ and proved the well-posedness of the corresponding continuum model in the large-sample limit. We remark that our continuum model is proposed for the case of finite number of data points, i.e., $n<+\infty$, not the large-sample limit case.


From extensive synthetic and realistic datasets, frequency principle is proposed to characterize the training process of deep neural networks \citep{xu_training_2018,rahaman2018spectral,xu2019frequency}. A series of theoretical works subsequently show that frequency principle holds in different settings, for example, a non-NTK (neural tangent kernel) regime with infinite samples \citep{luo2019theory} and the NTK regime with finite samples \citep{zhang_explicitizing_2019,bordelon2020spectrum,Luotao2020LFP} or infinite samples \citep{cao2019towards,basri2019convergence}. \cite{e2019machine} show that the integral equation would naturally leads to the frequency principle. The frequency principle inspires the design of deep neural networks to fast learn a function with high frequency \citep{liu2020multi,wang2020multi,jagtap2019adaptive,cai2019phasednn,s.20201019,li2020multi,wangbo2020multi}. 


\section{Fourier-domain Variational Problem for Supervised Learning}\label{section3}
\subsection{Motivation: Linear Frequency Principle}

In the following, we consider the regression problem of fitting a target function $f\in C_c(\sR^d)$. Clearly, $f\in L^{2}(\sR^d)$. Specifically,
we use a DNN, $h_{\mathrm{DNN}}(\vx,\vtheta(t))$,
to fit the training dataset $\{(\vx_{i},y_{i})\}_{i=1}^{n}$ of $n$ sample
points, where $\vx_{i}\in\sR^d$, $y_{i}=f(\vx_{i})$ for each $i$.
For the convenience of notation, we denote $\vX=(\vx_{1},\ldots,\vx_{n})^\T$,
$\vY=(y_{1},\ldots,y_{n})^\T$.
It has been shown in \citep{jacot2018neural,lee2019wide} that, if the number of neurons in each hidden layer is sufficiently large, then $\norm{\vtheta(t)-\vtheta(0)}\ll1$ for any $t\ge 0$. In such cases, the the following function
\begin{equation*}
h(\vx,\vtheta)=h_{\mathrm{DNN}}(\vx,\vtheta_{0})+\nabla_{\vtheta}h_{\mathrm{DNN}}\left(\vx,\vtheta_0\right)\cdot(\vtheta-\vtheta_{0}),\label{eq:linear}
\end{equation*}
is a very good approximation of DNN output $h_{\mathrm{DNN}}(\vx,\vtheta(t))$
with $\vtheta(0)=\vtheta_{0}$. Note that,
we have the following requirement for $h_{\mathrm{DNN}}$ which is
easily satisfied for common DNNs: for any $\vtheta\in\sR^{m}$,
there exists a weak derivative of $h_{\mathrm{DNN}}(\cdot,\vtheta_0)$
with respect to $\vtheta$ satisfying $\nabla_{\vtheta}h_{\mathrm{DNN}}(\cdot,\vtheta_0)\in L^{2}(\sR^d)$.

Inspired by the F-Principle and the linear dynamics in the kernel regime,
\citep{zhang_explicitizing_2019,Luotao2020LFP} derived a Linear F-Principle (LFP)
dynamics to effectively study the training dynamics of a two-layer ReLU NN with the mean square loss in the large width limit. Up to a multiplicative constant
in the time scale, the gradient
descent dynamics of a sufficiently wide two-layer NN is approximated by
\begin{equation}
\partial_{t}\fF[u](\vxi,t)=-\,(\gamma(\vxi))^{2}\fF[u_{\rho}](\vxi),\label{eq:ReLUnnFP}
\end{equation}
\noindent where $u(\vx,t)=h(\vx,t)-h_\mathrm{target}(\vx)$, $u_{\rho}(\vx)=u(\vx,t)\rho(\vx)$. We follow this work and further assume that $\rho(\vx)=\frac{1}{n}\sum_{i=1}^n\delta(\vx-\vx_i)$, accounting for the real case of a finite training dataset $\{(\vx_i,y_i)\}_{i=1}^n$, and
\begin{equation*}
(\gamma(\vxi))^{2}=\Exp_{a(0), r(0)}\left[\frac{r(0)^{3}}{16 \pi^{4}\norm{\vxi}^{d+3}}+\frac{a(0)^{2} r(0)}{4 \pi^{2}\norm{\vxi}^{d+1}}\right],
\end{equation*}
where $r(0)=\abs{\vw(0)}$ and the two-layer ReLU NN parameters at initial $a(0)$ and $\vw(0)$ are random variables with certain given distribution. 
In this work, for any function $g$ defined on $\sR^d$, we use the following convention of the Fourier transform and its inverse:
\begin{equation*}
\fF[g](\vxi)=\int_{\sR^d}g(\vx)\E^{-2\pi \I\vxi^\T\vx}\diff\vx,\quad g(\vx)=\int_{\sR^d}\fF[g](\vxi)\E^{2\pi \I\vx^\T\vxi}\diff\vxi.
\end{equation*}

\noindent Different from $\fF[u](\vxi,t)$ on the left hand side, the formula on the right hand side reads as
\begin{equation*}
\fF[u_\rho](\vxi,t)=\fF[u(\cdot,t)\rho(\cdot)](\vxi,t)=\frac{1}{n}\fF\left[\sum_{i=1}^{n}\left(h(\cdot,\vtheta(t))-y_{i}\right)\delta(\cdot-\vx_{i})\right](\vxi,t), 
\end{equation*}
which incorporates the information of the training dataset. The solution of the LFP model (\ref{eq:ReLUnnFP}) is equivalent to that of the following optimization problem in a proper hypothesis space $F_\gamma$,
\begin{equation*}
\min_{h-h_{\mathrm{ini}}\in F_{\gamma}}\int_{\sR^d}(\gamma(\vxi))^{-2}\abs{\fF[h](\vxi)-\fF[h_\mathrm{ini}](\vxi)}^{2}\diff{\vxi},\label{eq: minFPnorm-1}
\end{equation*}
subject to constraints $h(\vx_{i})=y_{i}$ for $i=1,\ldots,n$. The weight $(\gamma(\vxi))^{-2}$ grows as the frequency $\vxi$ increases, which means that a large penalty is imposed on the high frequency part of $h(\vx)-h_{\mathrm{ini}}(\vx)$. As we can see, a random non-zero initial output of DNN leads to a specific type of generalization error. To eliminate
this error, we use DNNs with an antisymmetrical initialization (ASI) trick ~\citep{zhang2020type}, which guarantees $h_{\mathrm{ini}}(\vx)=0$. Then the final output $h(\vx)$ is dominated by low frequency, and the DNN model possesses a good generalization.

\subsection{Fourier-domain Variational Formulation}

Inspired by the variational formulation of LFP model, we propose a new continuum model for the supervised learning. This is a variational problem with a parameter $\alpha>0$ to be determined later:
\begin{align}
& \min_{h\in \fH} Q_{\alpha}[h] =\int_{\sR^d}\langle\vxi\rangle^\alpha\Abs{\fF[h](\vxi)}^{2}\diff{\vxi},\label{mini-prob} \\
& \mathrm{s.t.}\quad h(\vx_i)=y_i,\quad i=1,\cdots,n,
\end{align}
where $\langle\vxi\rangle=(1+\norm{\vxi}^2)^{\frac{1}{2}}$ is the ``Japanese bracket'' of $\vxi$ and $\fH=\{h(x)|\int_{\sR^d}\langle\vxi\rangle^\alpha\Abs{\fF[h](\vxi)}^{2}\diff{\vxi}<\infty\}$.
Note that in the spatial domain, the evaluation on $n$ known data points is meaningless in the sense of $L^2$ functions. Therefore, we consider the problem in the frequency domain and define a linear operator $\fP_{\vX}:L^1(\sR^d)\cap L^2(\sR^d)\to\sR^n$ for the given sample set $\vX$ to transform the original constraints into the ones in the Fourier domain: $\fP_{\vX}\phi^*=\vY$. More precisely, we define for $\phi\in L^1(\sR^d)\cap L^2(\sR^d)$
\begin{equation}
\fP_{\vX}\phi:=\left(\int_{\sR^d}\phi(\vxi)\E^{2\pi\I\vxi\cdot \vx_{1}}\diff{\vxi},\cdots,
\int_{\sR^d}\phi(\vxi)\E^{2\pi\I\vxi\cdot \vx_{n}}\diff{\vxi}\right)^\T.
\end{equation}
\noindent The admissible function class reads as 
\begin{equation*}
    \fA_{\vX,\vY}=\{\phi\in L^1(\sR^d)\cap L^2(\sR^d)\mid\fP_{\vX}\phi=\vY\}.
\end{equation*}

Notice that
$\norm{\fF^{-1}[\phi]}_{H^{\frac{\alpha}{2}}}=\left(\int_{\sR^d}\langle\vxi\rangle^\alpha\abs{\phi(\vxi)}^{2}\diff{\vxi}\right)^{\frac{1}{2}}$ is a Sobolev norm, which characterizes the regularity of the final output function $h(\vx)=\fF^{-1}[{\phi}](\vx)$. The larger the exponent $\alpha$ is, the better the regularity becomes. 

For example, when $d=1$ and $\alpha=2$, by Parseval's theorem,
\begin{equation*}
    \norm{u}_{H^{1}}^2=\int_{\sR}(1+\abs{\xi}^2)\abs{\fF[u](\xi)}^{2}\diff{\xi}=\int_{\sR}u^2+\frac{1}{4\pi^2}\abs{\nabla u}^2\diff x.
\end{equation*}
Accordingly, the Fourier-domain variational problem reads as a standard variational problem in spatial domain. This is true for any quadratic Fourier-domain variational problem, but of course our Fourier-domain variational formulation is not necessarily being quadratic. The details for general cases (non-quadratic ones) are left to future work. For the quadratic setting with exponent $\alpha$, i.e., Problem~\eqref{mini-prob}, it is roughly equivalent to the following spatial-domain variational problem:
\begin{equation*}
    \min \int_{\sR^d}(u^2+\abs{\nabla ^{\frac{\alpha}{2}}u}^2)\diff x.
\end{equation*}
This is clear for integer $\alpha/2$, while fractional derivatives are required for non-integer $\alpha/2$. 

Back to our problem, after the above transformation, our goal is transformed into studying the following Fourier-domain variational problem,
\begin{prob}\label{prob..VariationalPointCloud}
	Find a minimizer $\phi^*$ in $\fA_{\vX,\vY}$ such that
	\begin{equation}
	 \phi^*\in\arg\min_{\phi\in \fA_{\vX,\vY}} \norm{\fF^{-1}[\phi]}_{H^{\frac{\alpha}{2}}}^2.
	\end{equation}
\end{prob}

\noindent This formulation is novel in the following two aspects:\\
1. We regard $\fF[h]$ as the primal solution.\\
2. The evaluation on the sample points $\vx_i$'s are imposed by the linear operator $\fP_{\vX}$.\\
Now we explain the importance of these new viewpoints for our task.

Traditionally, we all considered the problem in $\vx-y$ space, the spatial domain. Recently, by the understanding of F-Principle ~\citep{xu_training_2018,xu2019frequency,Luotao2020LFP,zhang_explicitizing_2019}, we believe that if DNN is used to fit the data, it is more natural to consider the problem in the frequency domain. In particular, the functions defined on Fourier domain are assumed to be primal. And our variational problem is asked for such functions.

We remark that the operator $\fP_{\vX}$ is the inverse Fourier transform with evaluations on sample points $\vX$. Actually, the linear operator $\fP_{\vX}$ projects a function defined on $\sR^d$ to a function defined on $0$-dimensional manifold $\vX$. Just like the (linear) trace operator $T$ in a Sobolev space projects a function defined on $d$-dimensional manifold into a function defined on $(d-1)$-dimensional boundary manifold. Note that the only function space over the $0$-dimensional manifold $\vX$ is the $n$-dimensional vector space $\sR^n$, where $n$ is the number of data points, while any Sobolev (or Besov) space over $d$-dimensional manifold ($d\geq 1$) is an infinite dimensional vector space.

\section{Existence and Non-existence of Fourier-domain Variantional Problems}\label{section4}

In this section, we consider the existence/non-existence dichotomy to Problem~\ref{prob..VariationalPointCloud}. In subsection 4.1, we prove that there is no solution to the Problem~\ref{prob..VariationalPointCloud} in subcritical case $\alpha < d$. The supercritical case $\alpha>d$ will be investigated in subsection 4.2, where we prove the optimal function is a continuous and nontrivial solution. All proof of propositions and theorems in this section can be found in Appendix. 
\subsection{Subcritical Case: $\alpha<d$}  
In order to prove the nonexistence of the solution to the Problem~\ref{prob..VariationalPointCloud} in $\alpha<d$ case, at first we need to find a class of functions that make the norm tend to zero. 
Let $\psi_{\sigma}(\vxi)=(2\pi)^{\frac{d}{2}}\sigma^d \E^{-2\pi^2\sigma^2\norm{\vxi}^2}$ , then by direct calculation, we have $\fF^{-1}[\psi_{\sigma}](\vx)=\E^{-\frac{\norm{\vx}^2}{2\sigma^2}}$. For $\alpha<d$ the following proposition shows that the norm $\norm{\fF^{-1}[\psi_\sigma]}_{H^{\frac{\alpha}{2}}}^2$ can be sufficiently small as $\sigma\rightarrow 0$.

\begin{prop}[critical exponent]\label{prop..CriticalExponent}
	For any input dimension $d$, we have
	\begin{equation}
	\lim_{\sigma\to0}\norm{\fF^{-1}[\psi_\sigma]}_{H^{\frac{\alpha}{2}}}^2
	=\begin{cases}
	0,      & \alpha<d, \\
	C_d,      & \alpha=d, \\
	\infty, & \alpha>d.
	\end{cases}
	\end{equation}
	Here the constant $C_d=\frac{1}{2}(d-1)!(2\pi)^{-d}\frac{2\pi^{d/2}}{\Gamma\left(d/2\right)}$ only depends on the dimension $d$.
\end{prop}

\begin{rmk}
	The function $\fF^{-1}[\psi]$ can be any function in the Schwartz space, not necessarily Gaussian. Proposition~\ref{prop..CriticalExponent} still holds with (possibly) different $C_d$.
\end{rmk}
For every small $\sigma$, we can use $n$ rapidly decreasing functions $\fF^{-1}[\psi_{\sigma}](\vx-\vx_{i})$ to construct the solution $\fF^{-1}[\phi_{\sigma}](\vx)$ of the supervised learning problem. However, according to Proposition~\ref{prop..CriticalExponent}, when the parameter $\sigma$ tends to 0, the limit is the zero function in the sense of $L^2(\sR^d)$. Therefore we have the following theorem:
\begin{thm}[non-existence]\label{thm..alpha<d} 
	Suppose that $\vY\neq\vzero$. For $\alpha<d$, there is no function $\phi^*\in \fA_{\vX,\vY}$ satisfying
	\begin{equation*}
	 \phi^*\in\arg\min_{\phi\in \fA_{\vX,\vY}}\norm{\fF^{-1}[\phi]}_{H^{\frac{\alpha}{2}}}^2.
	\end{equation*}
	In other words, there is no solution to the Problem~\ref{prob..VariationalPointCloud}. 
\end{thm}


\subsection{Supercritical Case: $\alpha>d$}\label{Supercritical case}
In this section, we provide a theorem to establish the existence of the minimizer for Problem \ref{prob..VariationalPointCloud} in the case of $\alpha>d$. 

\begin{thm}[existence]\label{thm..alpha>d} 
	For $\alpha>d$, there exists $\phi^*\in \fA_{\vX,\vY}$ satisfying
	\begin{equation*}
	 \phi^*\in\arg\min_{\phi\in \fA_{\vX,\vY}}\norm{\fF^{-1}[\phi]}_{H^{\frac{\alpha}{2}}}^2.
	\end{equation*}
	In other words, there exists a solution to the Problem \ref{prob..VariationalPointCloud}.
\end{thm}
\begin{rmk}
    Note that, according to the Sobolev embedding theorem~\citep{adams2003sobolev,1999Partial}, the minimizer in Theorem \ref{thm..alpha>d} has smoothness index no less than $[\frac{\alpha-d}{2}]$.
\end{rmk}


\section{Numerical Results}\label{section5}
In this section, we illustrate our results by solving Fourier-domain variational problems numerically. We use uniform mesh in frequency domain with mesh size $\Delta\xi$ and band limit $M\Delta\xi$. In this discrete setting, the considered space becomes $\sR^{(2M)^d}$. We emphasize that the numerical solution with this setup always exists even for the subcritical case which corresponds to the non-existence theorem. However, as we will show later, the numerical solution is trivial in nature when $\alpha<d$.

\subsection{Special Case: One Data Point in One Dimension}\label{case1}
To simplify the problem, we start with a single point $X=0\in\sZ$ with the label $Y=2$. Denote $\phi_j = \phi(\xi_j)$ for $j\in\sZ$. We also assume that the function $\phi$ is an even function. Then according to the definition of $\fP_{\vX}$, we have the following problem:
\begin{exam}[Problem~\ref{prob..VariationalPointCloud} with a particular discretization]\label{prob..BandLimit_discretization}  
	\begin{align}
	& \min_{\phi\in\sR^M} \sum_{j=1}^M(1+{j}^2\Delta\xi^2)^{\frac{\alpha}{2}}\Abs{\phi_j}^{2}, \\
	& \mathrm{s.t.}\quad \sum_{j=1}^M\phi_j\Delta\xi = 1,
	\end{align}
\end{exam}
\noindent where we further assume $\phi_0 = \phi(0) = 0$.  If we denote $\vphi = {(\phi_1, \phi_2, \ldots, \phi_M)}^{\T}$, $b = \frac{1}{\Delta\xi}$, $\mA = (1, 1, \ldots, 1)\in\sR^M$ and
\begin{equation*}
\mGamma = \sqrt{\lambda}
\begin{pmatrix}
(1+1^2\Delta\xi^2)^{\frac{\alpha}{4}} & & & \\
& (1+2^2\Delta\xi^2)^{\frac{\alpha}{4}} & & \\
& & \ddots & \\
& & & (1+M^2\Delta\xi^2)^{\frac{\alpha}{4}}
\end{pmatrix}.
\end{equation*}
In fact this is a standard Tikhonov regularization~\citep{1977Solutions} also known as ridge regression problem with the Lagrange multiplier $\lambda$. The corresponding ridge regression problem is,
\begin{equation}
\min_{\vphi}{\norm{\mA\vphi - b}_2^2 + \norm{\mGamma\vphi}_2^2},
\end{equation}
where we put $\lambda$ in the optimization term $\norm{\mGamma\vphi}_2^2$, instead of the constraint term $\norm{\mA\vphi - b}_2^2$. This problem admits an explicit and unique solution~\citep{1977Solutions},
\begin{equation}\label{ridge_solution}
\vphi = {(\mA^{\T}\mA + \mGamma^{\T}\mGamma)}^{-1}\mA^{\T} b.
\end{equation}
Here we need to point out that the above method is also applicable to the case that the matrix $\mGamma$ is not diagonal.

Back to our problem, in order to obtain the explicit expression for the optimal $\vphi$ we need the following relation between the solution of the ridge regression and the singular-value decomposition (SVD).

By denoting $\tilde{\mGamma} = \mI$ and
\begin{equation*}
\tilde{\mA} = \mA\mGamma^{-1}
= \frac{1}{\sqrt{\lambda}}\left( (1+1^2\Delta\xi^2)^{\frac{\alpha}{4}}, (1+2^2\Delta\xi^2)^{\frac{\alpha}{4}}, \ldots, (1+M^2\Delta\xi^2)^{\frac{\alpha}{4}} \right),
\end{equation*}
where $\mI$ is the diagonal matrix, the optimal solution~\eqref{ridge_solution} can be written as
\begin{equation*}
\vphi = {(\mGamma^{\T})}^{-1}{\left( \tilde{\mA}^{\T}\tilde{\mA} +\mI \right)}^{-1}\mGamma^{-1}\mA^{\T} b
= {(\mGamma^{\T})}^{-1}{\left( \tilde{\mA}^{\T}\tilde{\mA} +\mI \right)}^{-1}\tilde{\mA}^{\T} b
= {(\mGamma^{\T})}^{-1}\tilde{\vphi},
\end{equation*}
where $\tilde{\vphi}={\left( \tilde{\mA}^{\T}\tilde{\mA} +\mI \right)}^{-1}\tilde{\mA}^{\T} b$ is the solution of ridge regression with $\tilde{\mA}$ and $\tilde{\mGamma}$. In order to obtain the explicit expression for $\tilde{\vphi}$ we need the following relation between the solution of the ridge regression and the singular-value decomposition (SVD).

\begin{lem}\label{ridge_svd}  
	If $\tilde\mGamma = \mI$, then this least-squares solution can be solved using SVD. Given the singular value decomposition
	\begin{equation*}
	\tilde\mA = \mU\mSigma\mV^{\T},
	\end{equation*}
	with singular values $\sigma_i$, the Tikhonov regularized solution can be expressed aspects
	\begin{equation*}
	\tilde\vphi = \mV\mD\mU^{\T} b,
	\end{equation*}
	where $\mD$ has diagonal values
	\begin{equation*}
	D_{ii} = \frac{\sigma_i}{\sigma_i^2 + 1},
	\end{equation*}
	and is zero elsewhere.
\end{lem}
\begin{proof} 
	In fact, $\tilde\vphi={(\tilde\mA^{\T}\tilde\mA + \tilde\mGamma^{\T}\tilde\mGamma)}^{-1}\tilde\mA^{\T} b
	=\mV(\mSigma^\T\mSigma+1\vI)^{-1}\mV^\T\mV\mSigma^\T\mU^\T b\\
	=\mV\mD\mU^{\T} b$, which completes the proof.
\end{proof}

\noindent Since $\tilde{\mA}\tilde{\mA}^{\T} = \dfrac{1}{\lambda}\sum_{j = 1}^M (1+j^2\Delta\xi^2)^{-\frac{\alpha}{2}}$, we have $\tilde{\mA} = U \Sigma \vV^{\T}$ with
\begin{equation*}
U = 1, \quad \Sigma = \frac{1}{\sqrt{\lambda}}{\left( \sum_{j = 1}^M (1+j^2\Delta\xi^2)^{-\frac{\alpha}{2}} \right)}^{\frac{1}{2}} := {Z}/{\sqrt{\lambda}},
\end{equation*}
\begin{equation*}
\vV = {\left( (1+1^2\Delta\xi^2)^{-\frac{\alpha}{2}}/Z, (1+2^2\Delta\xi^2)^{-\frac{\alpha}{2}}/Z, \ldots, (1+M^2\Delta\xi^2)^{-\frac{\alpha}{2}}/Z \right)}^{\T}.
\end{equation*}
Then we get the diagonal value
\begin{equation*}
D = \frac{{Z}/{\sqrt{\lambda}}}{{Z^2}/{\lambda} + 1}.
\end{equation*}
Therefore, by Lemma ~\ref{ridge_svd}
\begin{equation*}
\tilde{\vphi} = \vV D U b = \dfrac{{1}/{\sqrt{\lambda}}}{{Z^2}/{\lambda} + 1}{\left( (1+1^2\Delta\xi^2)^{-\frac{\alpha}{2}},(1+2^2\Delta\xi^2)^{-\frac{\alpha}{2}}, \ldots, (1+M^2\Delta\xi^2)^{-\frac{\alpha}{2}} \right)}^{\T} b.
\end{equation*}
Finally, for the original optimal solution
\begin{align*}
\vphi   = {(\mGamma^{\T})}^{-1}\tilde{\vphi} =  \frac{1}{(Z^2 + \lambda)\Delta\xi}{\left( (1+1^2\Delta\xi^2)^{-\frac{\alpha}{2}}, (1+2^2\Delta\xi^2)^{-\frac{\alpha}{2}}, \ldots, (1+M^2\Delta\xi^2)^{-\frac{\alpha}{2}} \right)}^{\T},   
\end{align*}
which means
\begin{equation*}
\phi_j = \frac{(1+j^2\Delta\xi^2)^{-\frac{\alpha}{2}}}{(Z^2 + \lambda)\Delta\xi}.
\end{equation*}
To derive the function in $x$ space, say $h(x)$ then
\begin{align}\label{h(x)}
h(x) & = \frac{1}{(Z^2 + \lambda)} \sum_{j=-M}^{M} (1+j^2\Delta\xi^2)^{-\frac{\alpha}{2}} \E^{2\pi\I j x} \nonumber\\
& = \frac{2}{(Z^2 + \lambda)} \sum_{j=1}^{M} (1+j^2\Delta\xi^2)^{-\frac{\alpha}{2}} \cos(2\pi j x).
\end{align}

Fig.~\ref{fig:hx_is_nontrivial}  shows that for this special case with a large $M$, $h(x)$ is not an trivial function in $\alpha>d$ case and degenerates to a trivial function in $\alpha<d$ case.

\begin{figure} 
	\centering
    \includegraphics[width=0.5\textwidth]{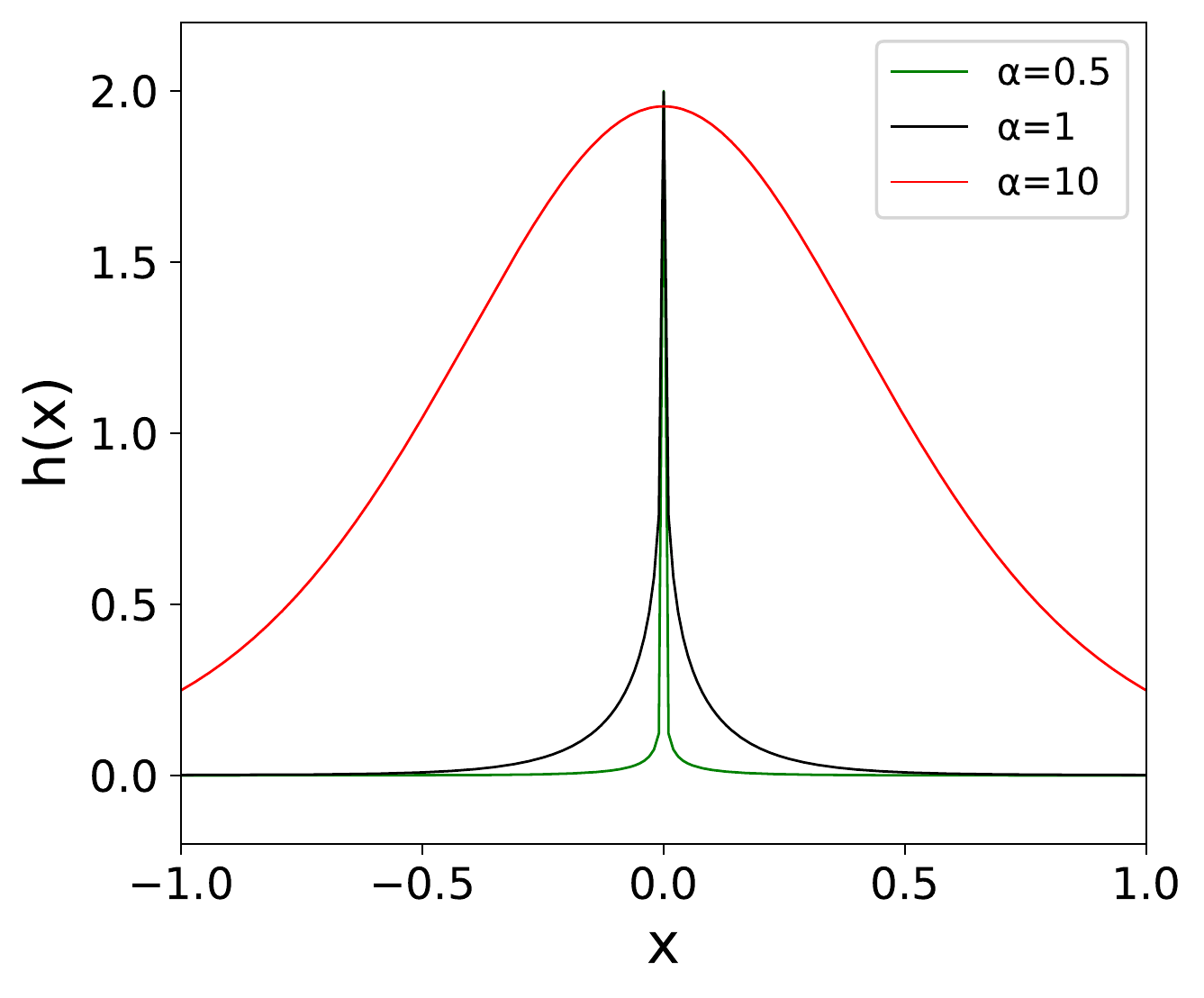}
	\caption{Fitting the function $h(x)$ shown in equation~\eqref{h(x)} with different exponent $\alpha$'s. Here we take $M=10^6$, $\Delta\xi=0.01$, $\lambda=1$ and different $\alpha$ and observe that $h(x)$ is not an trivial function in $\alpha>d$ case and degenerates to a trivial function in $\alpha<d$ case.}
	\label{fig:hx_is_nontrivial}
\end{figure}


\subsection{General Case: $n$ Points in $d$ Dimension}\label{case2}
Assume that we have $n$ data points $\vx_1,\vx_2,\ldots,\vx_n\in \sR^d$ and each data point has $d$ components:
\begin{equation*}
\vx_i=\left(x_{i1},x_{i2},\ldots,x_{id}\right)^\T
\end{equation*}
and denote the corresponding label as $\left(y_1,y_2,\ldots,y_n\right)^\T$. For the sake of simplicity, we denote the vector $(j_1,j_2,\cdots,j_d)^\T$ by $\vJ_{j_1\ldots j_d}$. Then our problem becomes
\begin{exam}[Problem \ref{prob..VariationalPointCloud} with general discretization]
	\begin{align}
	& \min_{\phi\in \sR^{(2M)^d}} \sum_{j_1,\ldots,j_d=-M}^M(1+\norm{\vJ_{j_1\ldots j_d}}^2\Delta\xi^2)^{\frac{\alpha}{2}}\Abs{\phi_{j_1\ldots j_d}}^{2}, \\
	& \mathrm{s.t.}\quad \sum_{j_1,\ldots,j_d=-M}^M\phi_{j_1\ldots j_d}\E^{2\pi\I\Delta\xi\vJ_{j_1\ldots j_d}^\T\vx_k}=y_k, \ \ k=1,2,\ldots,d
	\end{align}
\end{exam}

\noindent The calculation of this example can be completed by the method analogous to the one used in subsection~\ref{case1}. Let 
\begin{equation}\label{Aj}
\mA_j=\left(\E^{2\pi\I\Delta\xi\vJ_{-M-M\ldots -M}^\T\vx_j},\ldots, \E^{2\pi\I\Delta\xi\vJ_{j_1j_2\ldots j_d}^\T\vx_j},\ldots,\E^{2\pi\I\Delta\xi\vJ_{MM\ldots M}^\T\vx_j}\right)^\T,\ j=1,2,\ldots,n,
\end{equation}

\begin{equation}\label{A}
\mA=\left(\mA_1,\mA_2,\ldots,\mA_n \right)^\T\in \sR^{n\times (2M)^d},\quad
\vb=\left(y_1,y_2,\ldots,y_n\right)^\T\in \sR^{n\times 1},
\end{equation}
\begin{equation}\label{Gamma}
\mGamma=\lambda
\begin{pmatrix}
\ddots & &  \\
& (1+\norm{\vJ_{j_1j_2\ldots j_d}}^2\Delta\xi^2)^{\frac{\alpha}{4}}  & \\
& & \ddots 
\end{pmatrix}\in \sR^{(2M)^d\times(2M)^d}.
\end{equation}
We just need to solve the following equation:
\begin{equation}\label{phi}
\vphi = {(\mA^{\T}\mA + \mGamma^{\T}\mGamma)}^{-1}\mA^{\T} b.
\end{equation}
Then we can get the output function $h(x)$ by using inverse Fourier transform:
\begin{equation}\label{h}
h(\vx) = \sum_{j_1,\ldots,j_d=-M}^M\phi_{j_1\ldots j_d}\E^{2\pi\I\Delta\xi\vJ_{j_1\ldots j_d}\cdot\vx}
\end{equation}
Since the size of the matrix is too large, it is difficult to solve $\vphi$ by an explicit calculation. Thus we choose special $n$, $d$ and $M$ and show that $h(x)$ is not a trivial solution (non-zero function).

In our experiment, we set the hyper-parameter $M,\alpha,\lambda,\Delta\xi$ in advance. We set $\lambda=0,5,\Delta\xi=0.1$ in 1-dimensional case and $\lambda=0.2,\Delta\xi=0.1$ in 2-dimensional case. We select two data points $\{(-0.5,0.9),(0.5,0.9)\}$ as the given points in 1-dimensional case and four points as given points in 2-dimensional case whose second coordinates are 0.5 so that it is convenient to observe the phenomenon. At first, we use formula \eqref{Aj}, \eqref{A} and \eqref{Gamma} to calculate matrix $\mA, \mGamma$ and vector $\vb$. Then from the equation \eqref{phi} we can deduce vector $\vphi$. The final output function $h(\vx)$ is obtained by inverse discrete Fourier transform \eqref{h}.

 In Fig.\ref{fig:diff_M}, we set $\alpha=10$ in both cases to ensure $\alpha>d$ and change the band limit $M$. We observe that as $M$ increases, the fitting curve converges to a non-trivial curve. In Fig.\ref{fig:diff_alpha}, we set $M=1000$ in 1-dimensional case and $M=100$ in 2-dimensional case. By changing exponent $\alpha$, we can see in all cases, the fitting curves are non-trivial when $\alpha>d$, but degenerate when $\alpha<d$. 

%
%
\begin{figure}
	\subfigure[2 points in 1 dimension]{
	    \includegraphics[width=0.5\textwidth]{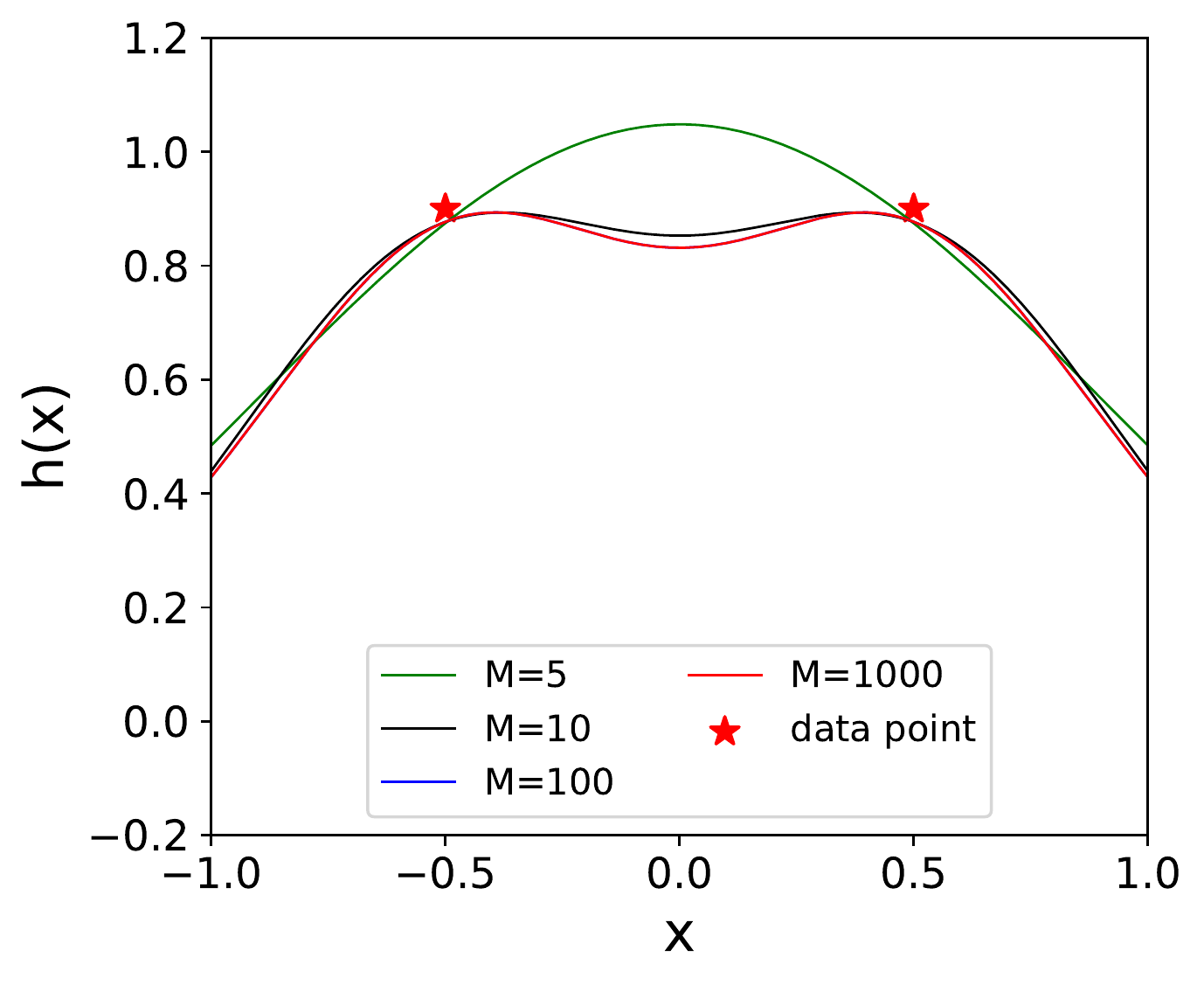}}
    \subfigure[20 points in 2 dimension]
        {\includegraphics[width=0.5\textwidth]{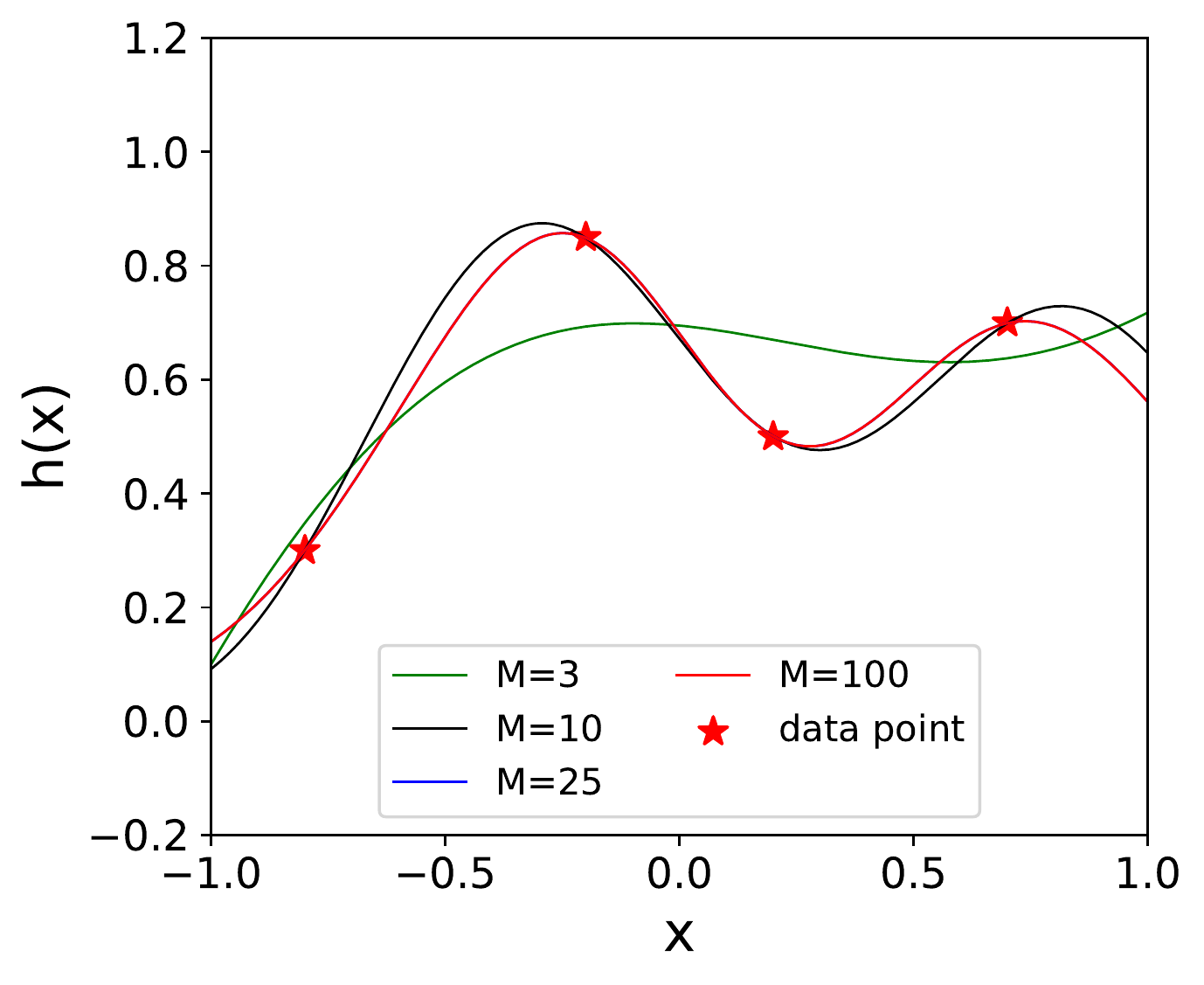}}
	\caption{Fitting data points in different dimensions with different band limit $M$. We use a proper $\alpha$ ($\alpha>d$) and observe that even for a large $M$, the function $h(x)$ does not degenerate to a trivial function. Note that the blue curve and the red one overlap with each. Here the trivial function represents a function whose value decays rapidly to zero except for the given training points.}
	\label{fig:diff_M}
\end{figure}

\begin{figure}
	\subfigure[2 points in 1 dimension]
	{
		\includegraphics[width=0.5\textwidth]{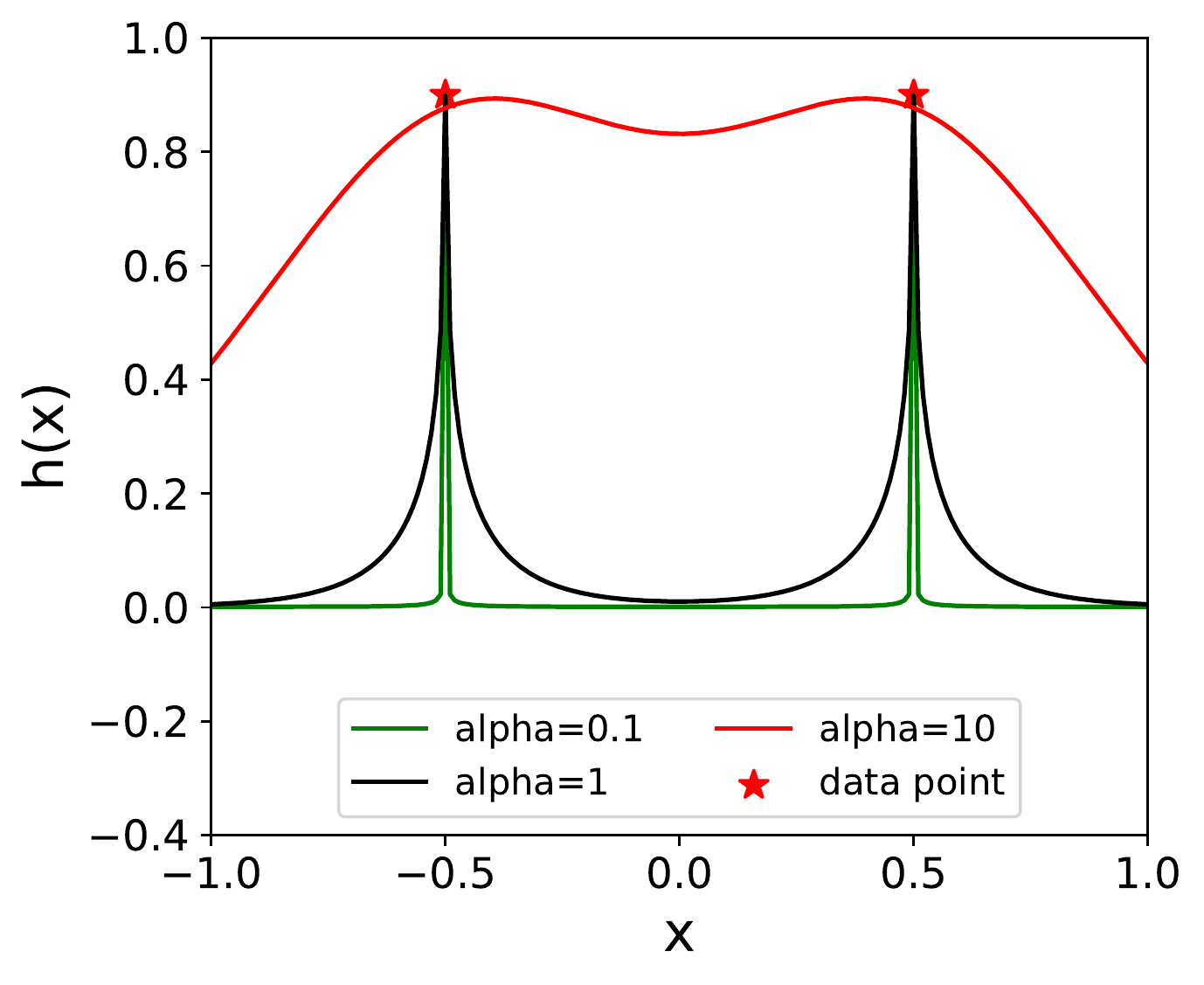}
	}
	\subfigure[20 points in 2 dimension]{
		\includegraphics[width=0.5\textwidth]{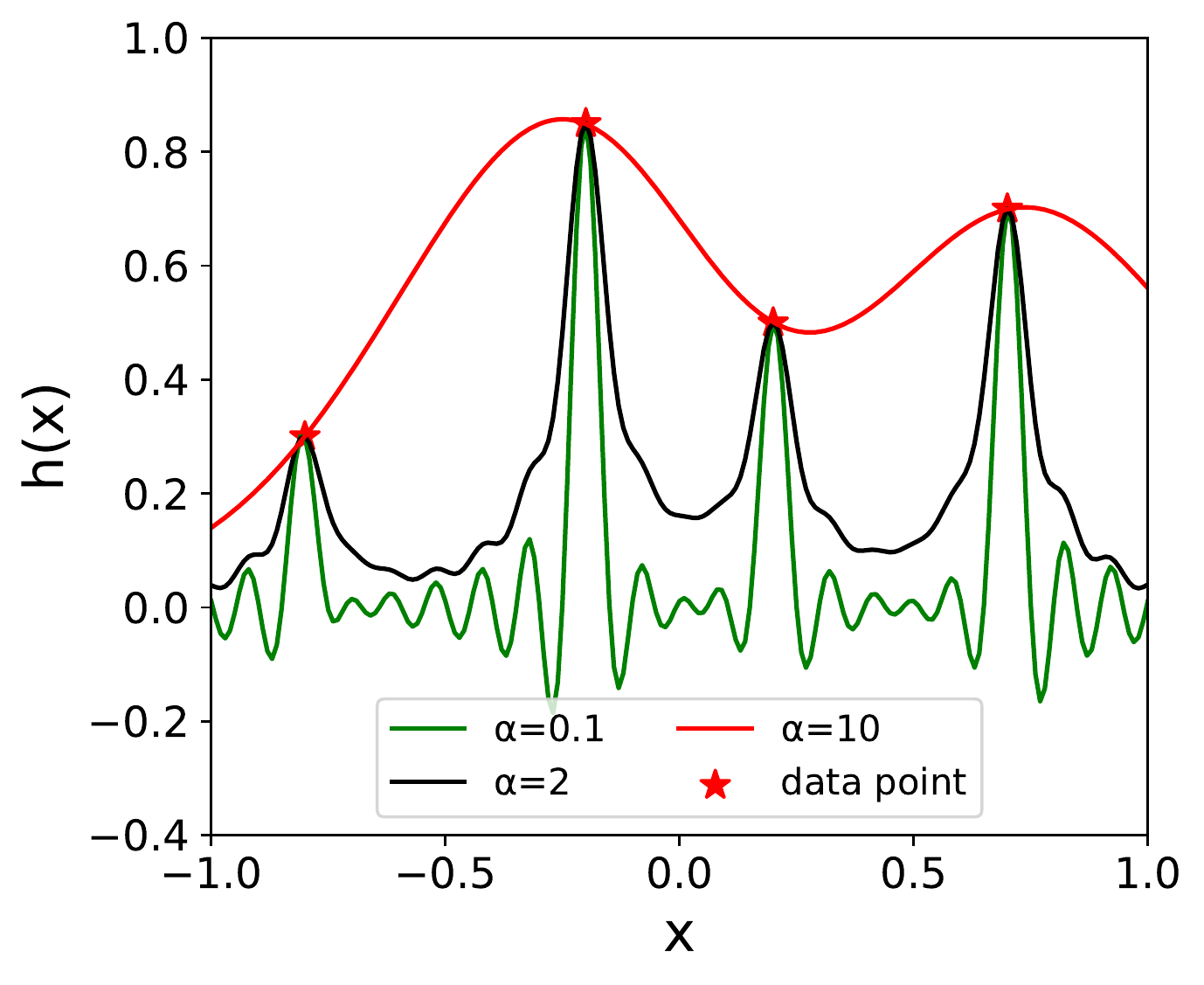}
	}
	\caption{Fitting data points in different dimensions with different exponent $\alpha$'s. We observe that with a proper $M$, the function $h(x)$ is not a trivial function for $\alpha>d$ case and degenerates to a trivial function for $\alpha>d$ case.}
	\label{fig:diff_alpha}
\end{figure}

\section{Conclusion \label{sec:conclusion}}
In this paper, we study the supervised learning problem by proposing a Fourier-domain variational formulation motivated by the frequency principle in deep learning. We establish the sufficient and necessary conditions for the well-posedness of the Fourier-domain variational problem, followed by numerical demonstration.

Our Fourier-domain variational formulation provides a novel viewpoint for modelling machine learning problem, that is, imposing more constraints, e.g., higher regularity, on the model rather than the data (always isolated points in practice) can give us the well-posedness as dimension of the problem increases. This is different from the modelling in physics and traditional point cloud problems, in which the model is independent of dimension in general. Our work suggests a potential approach of algorithm design by considering a dimension-dependent model for data modelling. 

In contrast to the natural sciences, where models are usually derived from fundamental physical laws, in data science, the story may be totally different: we may propose mathematical models based on the algorithms of great practical success like DNNs. Afterwards, the continuum formulation and its mathematical validation can be analyzed. This seems to be a new scientific paradigm, along which our work plays a role as one step.


\bibliography{DLRef_1}

\begin{thebibliography}{28}
\providecommand{\natexlab}[1]{#1}
\providecommand{\url}[1]{\texttt{#1}}
\expandafter\ifx\csname urlstyle\endcsname\relax
  \providecommand{\doi}[1]{doi: #1}\else
  \providecommand{\doi}{doi: \begingroup \urlstyle{rm}\Url}\fi

\bibitem[Adams and Fournier(2003)]{adams2003sobolev}
Robert~A. Adams and John~J.F. Fournier.
\newblock \emph{Sobolev spaces}.
\newblock Elsevier Science, 2003.

\bibitem[Biland et~al.(2020)Biland, Azevedo, Kim, and Solenthaler]{s.20201019}
Simon Biland, Vinicius~C. Azevedo, Byungsoo Kim, and Barbara Solenthaler.
\newblock {Frequency-aware reconstruction of fluid simulations with generative
  networks}.
\newblock In Alexander Wilkie and Francesco Banterle, editors,
  \emph{Eurographics 2020 - Short Papers}. The Eurographics Association, 2020.
\newblock ISBN 978-3-03868-101-4.
\newblock \doi{10.2312/egs.20201019}.

\bibitem[Bordelon et~al.(2020)Bordelon, Canatar, and
  Pehlevan]{bordelon2020spectrum}
Blake Bordelon, Abdulkadir Canatar, and Cengiz Pehlevan.
\newblock Spectrum dependent learning curves in kernel regression and wide
  neural networks.
\newblock \emph{arXiv preprint arXiv:2002.02561}, 2020.

\bibitem[Cai et~al.(2019)Cai, Li, and Liu]{cai2019phasednn}
Wei Cai, Xiaoguang Li, and Lizuo Liu.
\newblock A phase shift deep neural network for high frequency approximation
  and wave problems.
\newblock \emph{Accepted by SISC, arXiv:1909.11759}, 2019.

\bibitem[Calder and Slepev(2019)]{2019Properly}
Jeff Calder and Dejan Slepev.
\newblock Properly-weighted graph {Laplacian} for semi-supervised learning.
\newblock \emph{Applied Mathematics and Optimization}, \penalty0 (4), 2019.

\bibitem[Cao et~al.(2019)Cao, Fang, Wu, Zhou, and Gu]{cao2019towards}
Yuan Cao, Zhiying Fang, Yue Wu, Ding-Xuan Zhou, and Quanquan Gu.
\newblock Towards understanding the spectral bias of deep learning.
\newblock \emph{arXiv preprint arXiv:1912.01198}, 2019.

\bibitem[E et~al.(2020)E, Ma, and Wu]{e2019machine}
Weinan E, Chao Ma, and Lei Wu.
\newblock Machine learning from a continuous viewpoint, \text{I}.
\newblock \emph{Science China Mathematics}, pages 1--34, 2020.

\bibitem[El~Alaoui et~al.(2016)El~Alaoui, Cheng, Ramdas, Wainwright, and
  Jordan]{el2016asymptotic}
Ahmed El~Alaoui, Xiang Cheng, Aaditya Ramdas, Martin~J Wainwright, and
  Michael~I Jordan.
\newblock Asymptotic behavior of $l_p$-based \text{Laplacian} regularization in
  semi-supervised learning.
\newblock In \emph{Conference on Learning Theory}, pages 879--906, 2016.

\bibitem[Evans(1999)]{1999Partial}
Lawrence~C. Evans.
\newblock Partial differential equations.
\newblock \emph{Mathematical Gazette}, 83\penalty0 (496):\penalty0 185, 1999.

\bibitem[Jacot et~al.(2018)Jacot, Gabriel, and Hongler]{jacot2018neural}
Arthur Jacot, Franck Gabriel, and Cl{\'e}ment Hongler.
\newblock Neural tangent kernel: convergence and generalization in neural
  networks.
\newblock In \emph{Advances in neural information processing systems}, pages
  8571--8580, 2018.

\bibitem[Jagtap et~al.(2020)Jagtap, Kawaguchi, and
  Karniadakis]{jagtap2019adaptive}
Ameya~D Jagtap, Kenji Kawaguchi, and George~Em Karniadakis.
\newblock Adaptive activation functions accelerate convergence in deep and
  physics-informed neural networks.
\newblock \emph{Journal of Computational Physics}, 404:\penalty0 109136, 2020.

\bibitem[Lee et~al.(2019)Lee, Xiao, Schoenholz, Bahri, Novak, Sohl-Dickstein,
  and Pennington]{lee2019wide}
Jaehoon Lee, Lechao Xiao, Samuel Schoenholz, Yasaman Bahri, Roman Novak, Jascha
  Sohl-Dickstein, and Jeffrey Pennington.
\newblock Wide neural networks of any depth evolve as linear models under
  gradient descent.
\newblock In \emph{Advances in neural information processing systems}, pages
  8572--8583, 2019.

\bibitem[Li et~al.(2020)Li, Xu, and Zhang]{li2020multi}
Xi-An Li, Zhi-Qin~John Xu, and Lei Zhang.
\newblock A multi-scale \text{DNN} algorithm for nonlinear elliptic equations
  with multiple scales.
\newblock \emph{Communications in Computational Physics}, 28\penalty0
  (5):\penalty0 1886--1906, 2020.

\bibitem[Liu et~al.(2020)Liu, Cai, and Xu]{liu2020multi}
Ziqi Liu, Wei Cai, and Zhi-Qin~John Xu.
\newblock Multi-scale deep neural network (\text{MscaleDNN}) for solving
  {Poisson-Boltzmann} equation in complex domains.
\newblock \emph{Communications in Computational Physics}, 28\penalty0
  (5):\penalty0 1970--2001, 2020.

\bibitem[Luo et~al.(2019)Luo, Ma, Xu, and Zhang]{luo2019theory}
Tao Luo, Zheng Ma, Zhi-Qin~John Xu, and Yaoyu Zhang.
\newblock Theory of the frequency principle for general deep neural networks.
\newblock \emph{arXiv preprint arXiv:1906.09235}, 2019.

\bibitem[Luo et~al.(2020)Luo, Ma, Xu, and Zhang]{Luotao2020LFP}
Tao Luo, Zheng Ma, Zhi-Qin~John Xu, and Yaoyu Zhang.
\newblock On the exact computation of linear frequency principle dynamics and
  its generalization.
\newblock \emph{arXiv preprint arXiv:2010.08153}, 2020.

\bibitem[Nadler et~al.(2009)Nadler, Srebro, and
  Zhou]{Nadler2009SemisupervisedLW}
B.~Nadler, Nathan Srebro, and Xueyuan Zhou.
\newblock Semi-supervised learning with the graph {Laplacian}: the limit of
  infinite unlabelled data.
\newblock In \emph{NIPS 2009}, 2009.

\bibitem[Rahaman et~al.(2019)Rahaman, Baratin, Arpit, Draxler, Lin, Hamprecht,
  Bengio, and Courville]{rahaman2018spectral}
Nasim Rahaman, Aristide Baratin, Devansh Arpit, Felix Draxler, Min Lin, Fred
  Hamprecht, Yoshua Bengio, and Aaron Courville.
\newblock On the {spectral} {bias} of {neural} {networks}.
\newblock In \emph{International {Conference} on {Machine} {Learning}}, pages
  5301--5310, 2019.

\bibitem[Ronen et~al.(2019)Ronen, Jacobs, Kasten, and
  Kritchman]{basri2019convergence}
Basri Ronen, David Jacobs, Yoni Kasten, and Shira Kritchman.
\newblock The convergence rate of neural networks for learned functions of
  different frequencies.
\newblock In \emph{Advances in Neural Information Processing Systems}, pages
  4763--4772, 2019.

\bibitem[Shi et~al.(2017)Shi, Osher, and Zhu]{2017Weighted}
Zuoqiang Shi, Stanley Osher, and Wei Zhu.
\newblock Weighted nonlocal \text{Laplacian} on interpolation from sparse data.
\newblock \emph{Journal of Scientific Computing}, 73\penalty0 (2-3):\penalty0
  1--14, 2017.

\bibitem[Tikhonov and Arsenin(1977)]{1977Solutions}
Andrej~Nikolaevich Tikhonov and Vasiliy~Yakovlevich Arsenin.
\newblock Solutions of ill-posed problems.
\newblock \emph{Mathematics of Computation}, 32\penalty0 (144):\penalty0
  491--491, 1977.

\bibitem[Wang et~al.(2020{\natexlab{a}})Wang, Zhang, and Cai]{wangbo2020multi}
Bo~Wang, Wenzhong Zhang, and Wei Cai.
\newblock Multi-scale deep neural network (mscalednn) methods for oscillatory
  stokes flows in complex domains.
\newblock \emph{Communications in Computational Physics}, 28\penalty0
  (5):\penalty0 2139--2157, 2020{\natexlab{a}}.

\bibitem[Wang et~al.(2020{\natexlab{b}})Wang, Eljarrat, M{\"u}ller, Henninen,
  Erni, and Koch]{wang2020multi}
Feng Wang, Alberto Eljarrat, Johannes M{\"u}ller, Trond~R Henninen, Rolf Erni,
  and Christoph~T Koch.
\newblock Multi-resolution convolutional neural networks for inverse problems.
\newblock \emph{Scientific reports}, 10\penalty0 (1):\penalty0 1--11,
  2020{\natexlab{b}}.

\bibitem[Xu et~al.(2019)Xu, Zhang, and Xiao]{xu_training_2018}
Zhi-Qin~John Xu, Yaoyu Zhang, and Yanyang Xiao.
\newblock Training behavior of deep neural network in frequency domain.
\newblock \emph{International Conference on Neural Information Processing},
  pages 264--274, 2019.

\bibitem[Xu et~al.(2020)Xu, Zhang, Luo, Xiao, and Ma]{xu2019frequency}
Zhi-Qin~John Xu, Yaoyu Zhang, Tao Luo, Yanyang Xiao, and Zheng Ma.
\newblock Frequency principle: Fourier analysis sheds light on deep neural
  networks.
\newblock \emph{Communications in Computational Physics}, 28\penalty0
  (5):\penalty0 1746--1767, 2020.

\bibitem[Zhang et~al.(2019)Zhang, Xu, Luo, and Ma]{zhang_explicitizing_2019}
Yaoyu Zhang, Zhi-Qin~John Xu, Tao Luo, and Zheng Ma.
\newblock Explicitizing an implicit bias of the frequency principle in
  two-layer neural networks.
\newblock \emph{arXiv:1905.10264}, May 2019.

\bibitem[Zhang et~al.(2020)Zhang, Xu, Luo, and Ma]{zhang2020type}
Yaoyu Zhang, Zhi-Qin~John Xu, Tao Luo, and Zheng Ma.
\newblock A type of generalization error induced by initialization in deep
  neural networks.
\newblock In Jianfeng Lu and Rachel Ward, editors, \emph{Proceedings of The
  First Mathematical and Scientific Machine Learning Conference}, volume 107,
  pages 144--164, 2020.

\bibitem[Zhu et~al.(2003)Zhu, Ghahramani, and Lafferty]{zhu2003semi}
Xiaojin Zhu, Zoubin Ghahramani, and John~D Lafferty.
\newblock Semi-supervised learning using {Gaussian} fields and harmonic
  functions.
\newblock In \emph{Proceedings of the 20th International conference on Machine
  learning (ICML-03)}, pages 912--919, 2003.

\end{thebibliography}

\appendix

\section{Lemma~\ref{lemma1}}
\begin{lem}\label{lemma1}
	Let the function $\psi_{\sigma}(\vxi)=(2\pi)^{\frac{d}{2}}\sigma^d \E^{-2\pi^2\sigma^2\norm{\vxi}^2}$, $\vxi\in\sR^d$. We have
	\begin{equation}\label{lemmaeq}
	\lim_{\sigma\to0}\int_{\sR^d}\norm{\vxi}^{\alpha}\abs{\psi_{\sigma}(\vxi)}^{2}\diff{\vxi}
	=\begin{cases}
	0,      & \alpha<d, \\
	C_d,      & \alpha=d, \\
	\infty, & \alpha>d.
	\end{cases}
	\end{equation}
	Here the constant $C_d=\frac{1}{2}(d-1)!(2\pi)^{-d}\frac{2\pi^{d/2}}{\Gamma\left(d/2\right)}$ only depends on the dimension $d$.
\end{lem}
\begin{proof} 
	In fact, 
	\begin{align*}
	\lim\limits_{\sigma\rightarrow 0}  \int_{\sR^d}\norm{\vxi}^\alpha|\psi_\sigma(\vxi)|^2\diff\vxi
	&=\lim\limits_{\sigma\rightarrow 0}  \int_{\sR^d}\norm{\vxi}^\alpha(2\pi)^d\sigma^{2d}\E^{-4\pi^2\sigma^2\norm{\vxi}^2}\diff\vxi\\
	&=\lim\limits_{\sigma\rightarrow 0}  (2\pi)^d  \sigma^{d-\alpha}  \int_{\sR^d}  \norm{\sigma\vxi}^\alpha  \E^{-4\pi^2\norm{\sigma\vxi}^2}\diff{(\sigma\vxi)}\\
	&=\lim\limits_{\sigma\rightarrow 0}  (2\pi)^d  \sigma^{d-\alpha}  \int_0^\infty r^{\alpha+d-1}  \E^{-4\pi^2r^2}\diff r\cdot \omega_d,
	\end{align*}
	where $\omega_d=\frac{2\pi^{\frac{d}{2}}}{\Gamma\left(\frac{d}{2}\right)}$ is the surface area of a unit $(d-1)$-sphere.
	
	\noindent Notice that 
	\begin{align*}
	\int_0^\infty r^{\alpha+d-1}  \E^{-4\pi^2r^2}\diff r&= 
	\int_0^1 r^{\alpha+d-1}  \E^{-4\pi^2r^2}\diff r + \int_1^\infty r^{\alpha+d-1}  \E^{-4\pi^2r^2}\diff r\\
	&\le \int_0^\infty \E^{-4\pi^2r^2}\diff r + \int_0^\infty r^{[\alpha]+d}  \E^{-4\pi^2r^2}\diff r\\
	&=\frac{1}{8\pi^{\frac{3}{2}}} + \int_0^\infty r^{[\alpha]+d}  \E^{-4\pi^2r^2}\diff r
	\end{align*}
	and
	\begin{equation*}
	\int_0^\infty r^{[\alpha]+d}  \E^{-4\pi^2r^2}\diff r=
	\begin{cases}
	\frac{1}{2}\left(\frac{[\alpha]+d-1}{2}\right)!(2\pi)^{-([\alpha]+d+1)}, & [\alpha]+d\mathrm{\ is\ odd},\\
	\frac{\sqrt{\pi}}{2}(2\pi)^{-([\alpha]+d+1)}(\frac{1}{2})^{\frac{[\alpha]+d}{2}}([\alpha]+d-1)!!, & [\alpha]+d\mathrm{\ is\ even}.
	\end{cases}
	\end{equation*}
	Therefore, in both cases, the integral $\int_0^\infty r^{\alpha+d-1}  \E^{-4\pi^2r^2}\diff r$ is finite. Then we have 
	\begin{align*}
	\lim\limits_{\sigma\rightarrow 0}  \int_{\sR^d}\norm{\vxi}^\alpha|\psi_\sigma(\vxi)|^2\diff\vxi&=\lim\limits_{\sigma\rightarrow 0}  (2\pi)^d  \sigma^{d-\alpha}  \int_0^\infty r^{\alpha+d-1}  \E^{-4\pi^2r^2}\diff r\cdot \omega_d\\
	&=\begin{cases}
	0,      & \alpha<d, \\
	\infty, & \alpha>d.
	\end{cases}
	\end{align*}
	
	\noindent When $\alpha=d$, it follows that
	\begin{equation*}
	\int_0^\infty r^{\alpha+d-1}  \E^{-4\pi^2r^2}\diff r
	=\frac{1}{2}(2\pi)^{-2d}(d-1)!.
	\end{equation*}
	
	\noindent Therefore
	\begin{equation*}
	\lim\limits_{\sigma\rightarrow 0}  \int_{\sR^d}\norm{\vxi}^\alpha|\psi_\sigma(\xi)|^2\diff\xi
	=\frac{1}{2}(d-1)!(2\pi)^{-d}\frac{2\pi^{\frac{d}{2}}}{\Gamma\left(\frac{d}{2}\right)},
	\end{equation*}
	which completes the proof.
\end{proof}

\section{Proof of Proposition~\ref{prop..CriticalExponent}}

\begin{proof}  
	Similar to the proof of Lemma \ref{lemma1}, we have
	\begin{align*}
	\lim\limits_{\sigma\rightarrow 0}  \norm{\fF^{-1}[\psi_\sigma]}_{H^{\frac{\alpha}{2}}}^2
	&=\lim\limits_{\sigma\rightarrow 0}  (2\pi)^d  \sigma^{d-\alpha}  \int_{\sR^d}  (\sigma^2+\norm{\sigma\vxi}^2)^{\frac{\alpha}{2}}  \E^{-4\pi^2\norm{\sigma\vxi}^2}\diff(\sigma\vxi)\\
	&=\lim\limits_{\sigma\rightarrow 0}  (2\pi)^d  \sigma^{d-\alpha}  \int_0^\infty  r^{d-1}(\sigma^2+r^2)^{\frac{\alpha}{2}}  \E^{-4\pi^2r^2}\diff r\cdot \omega_d.\\
	\end{align*}
	For $\sigma<1$, the following integrals are bounded from below and above, respectively:
	\begin{equation*}
	\int_0^\infty  r^{d-1}(\sigma^2+r^2)^{\frac{\alpha}{2}}  \E^{-4\pi^2r^2}\diff r \ge \int_0^\infty  r^{\alpha+d-1}  \E^{-4\pi^2r^2}\diff r = C_1 >0,
	\end{equation*}
	and
	\begin{align*} 
	\int_0^\infty  r^{d-1}(\sigma^2+r^2)^{\frac{\alpha}{2}}  \E^{-4\pi^2r^2}\diff r &\le
	\int_0^1 r^{d-1}(1+r^2)^{\frac{\alpha}{2}}  \E^{-4\pi^2r^2}\diff r + 
	\int_1^\infty r^{d-1}((2r)^2)^{\frac{\alpha}{2}}  \E^{-4\pi^2r^2}\diff r \\&\le
	\int_0^1 r^{d-1}(1+r^2)^{\frac{\alpha}{2}}  \E^{-4\pi^2r^2}\diff r +
	2^\alpha\int_0^\infty r^{\alpha+d-1}  \E^{-4\pi^2r^2}\diff r\\&= C_2 < \infty,
	\end{align*}
	where $C_1=\int_0^\infty  r^{\alpha+d-1}  \E^{-4\pi^2r^2}\diff r$ and $C_2=\int_0^1 r^{d-1}(1+r^2)^{\frac{\alpha}{2}}  \E^{-4\pi^2r^2}\diff r +
	2^\alpha\int_0^\infty r^{\alpha+d-1}  \E^{-4\pi^2r^2}\diff r$. Therefore, we obtain the results for the subcritical ($\alpha<d$) and supercritical ($\alpha>d$) cases
	\begin{align*}
	\lim\limits_{\sigma\rightarrow 0}  \norm{\fF^{-1}[\psi_\sigma]}_{H^{\frac{\alpha}{2}}}^2
	&=\lim\limits_{\sigma\rightarrow 0}  (2\pi)^d  \sigma^{d-\alpha}  \int_0^\infty  r^{d-1}(\sigma^2+r^2)^{\frac{\alpha}{2}}  \E^{-4\pi^2r^2}\diff r\cdot \omega_d\\
	&=\begin{cases}
	0,      & \alpha<d, \\
	\infty, & \alpha>d.
	\end{cases}
	\end{align*}
	
	\noindent For the critical case $\alpha=d$, we have
	\begin{align*}
	&\lim\limits_{\sigma\rightarrow 0}  \norm{\fF^{-1}[\psi_\sigma]}_{H^{\frac{\alpha}{2}}}^2\\
	&=\lim\limits_{\sigma\rightarrow 0}  (2\pi)^d \int_0^\infty  r^{d-1}(\sigma^2+r^2)^{\frac{\alpha}{2}}  \E^{-4\pi^2r^2}\diff r\cdot \omega_d\\
	&=\lim\limits_{\sigma\rightarrow 0}  (2\pi)^d \int_0^\infty  r^{2d-1}  \E^{-4\pi^2r^2}\diff r\cdot \omega_d + \lim\limits_{\sigma\rightarrow 0}\left[\frac{\alpha}{2}(2\pi)^d \sigma^2\int_0^\infty r^{2d-3}\E^{-4\pi^2r^2}\diff r\cdot \omega_d+o(\sigma^2)\right]\\
	&=\lim\limits_{\sigma\rightarrow 0}  (2\pi)^d \int_0^\infty  r^{2d-1}  \E^{-4\pi^2r^2}\diff r\cdot \omega_d\\
	&=\frac{1}{2}(d-1)!(2\pi)^{-d}\frac{2\pi^{\frac{d}{2}}}{\Gamma\left(\frac{d}{2}\right)}.
	\end{align*}
	\noindent Therefore the proposition holds.
\end{proof}

\section{Proof of Theorem~\ref{thm..alpha<d}}

\begin{proof} 
	Given $\vX=(\vx_{1},\ldots,\vx_{n})^\T$ and $\vY=(y_{1},\ldots,y_{n})^\T$, let $\mA=\left(\exp(-\frac{\norm{\vx_j-\vx_i}^2}{2\sigma^2})\right)_{n\times n}$ be an $n\times n$ matrix. For sufficiently small $\sigma$, the matrix $\mA$ is diagonally dominant, and hence invertible. So the linear system $\mA\vg^{(\sigma)}=\vY$ has a solution $\vg^{(\sigma)}=\left(g^{(\sigma)}_1,g^{(\sigma)}_2,\cdots,g^{(\sigma)}_n\right)^\T$. Let
	\begin{equation*}
	\phi_{\sigma}(\vxi)
	=\sum_{i}g^{(\sigma)}_{i}\E^{-2\pi\I\vxi^\T\vx_i}\psi_{\sigma}(\vxi),
	\end{equation*}
	\noindent where $\psi_{\sigma}(\vxi)=(2\pi)^{\frac{d}{2}}\sigma^d\E^{-2\pi^2\sigma^2\norm{\vxi}^2}$ satisfying $\fF^{-1}[\psi_{\sigma}](\vx)=\E^{-\frac{\norm{\vx}^2}{2\sigma^2}}$. Thus
	\begin{equation*}
	\fF^{-1}[\phi_{\sigma}](\vx)
	=\sum_{i}g^{(\sigma)}_{i}\fF^{-1}[\psi_{\sigma}](\vx-\vx_{i})
	=\sum_{i}g^{(\sigma)}_{i}\E^{-\frac{\norm{\vx-\vx_{i}}^2}{2\sigma^2}}.
	\end{equation*}
	In particular, for all $i=1,2,\cdots,n$
	\begin{equation*}
	\fF^{-1}[\phi_{\sigma}](\vx_i)
	=\sum_{j}g^{(\sigma)}_{j}\E^{-\frac{\norm{\vx_i-\vx_{j}}^2}{2\sigma^2}}=(\mA\vg^{(\sigma)})_i=y_i.
	\end{equation*}
	\noindent Therefore, $\phi_{\sigma}\in \fA_{\vX,\vY}$ for sufficiently small $\sigma>0$.
	
	According to the above discussion, we can construct a sequence $\{\phi_{\frac{1}{m}}\}_{m=M}^\infty\subset \fA_{\vX,\vY}$, where $M$ is a sufficiently large positive integer to make the matrix $\mA$ invertible. As Proposition~\ref{prop..CriticalExponent} shows,
	\begin{equation*}
	\lim_{m\to+\infty}\norm{\fF^{-1}[\phi_{\frac{1}{m}}]}_{H^{\frac{\alpha}{2}}}^2=0.
	\end{equation*}
	Now, suppose that there exists a solution to the Problem \ref{prob..VariationalPointCloud}, denoted as $\phi^*\in \fA_{\vX,\vY}$. By definition,
	\begin{equation*}
    	\norm{\fF^{-1}[\phi^*]}_{H^{\frac{\alpha}{2}}}^2
    	\leq \min_{\phi\in \fA_{\vX,\vY}} \norm{\fF^{-1}[\phi]}_{H^{\frac{\alpha}{2}}}^2
    	\le\lim_{m\to+\infty}\norm{\fF^{-1}[\phi_{\frac{1}{m}}]}_{H^{\frac{\alpha}{2}}}^2=0.
	\end{equation*}
	Therefore, $\phi^*(\vxi)\equiv 0$ and $\fP_{\vX}\phi^*=\vzero$, which contradicts to the restrictive condition $\fP_{\vX}\phi^*=\vY$ for the situation that $\vY\neq \vzero$. The proof is completed.
\end{proof}

\section{Proof of Theorem~\ref{thm..alpha>d}}

\begin{proof}    
    \noindent 1.
    We introduce a distance for functions $\phi,\psi\in L^2(\sR^d)$:
    \begin{equation*}
        \dist(\phi,\psi)=\norm{\fF^{-1}[\phi]-\fF^{-1}[\psi]}_{H^{\frac{\alpha}{2}}}.
    \end{equation*}
    Under the topology induced by this distance, the closure of the admissible function class $\fA_{\vX,\vY}$ reads as
    \begin{equation*}
        \overline{\fA_{\vX,\vY}}:=\overline{\{\phi\in L^1(\sR^d)\cap L^2(\sR^d)\mid\fP_{\vX}\phi=\vY\}}^{\mathrm{dist}(\cdot,\cdot)}.
    \end{equation*}
    \noindent 2.
    We will consider an auxiliary minimization problem: to find $\phi^*$ such that
    \begin{equation}\label{eq..AuxiliaryMinProb}
        \phi^*\in\arg\min_{\phi\in \overline{\fA_{\vX,\vY}}}\norm{\fF^{-1}[\phi]}_{H^{\frac{\alpha}{2}}}.
    \end{equation}
    Let $m:=\inf_{\phi\in \overline{\fA_{\vX,\vY}}}\norm{\fF^{-1}[\phi]}_{H^{\frac{\alpha}{2}}}$.
	According to the proof of Proposition~\ref{prop..CriticalExponent} and Theorem~\ref{thm..alpha<d}, for a small enough $\sigma>0$, the inverse Fourier transform of function 
	\begin{equation*}
	\phi_{\sigma}(\vxi)
	=\sum_{i}g^{(\sigma)}_{i}\E^{-2\pi\I\vxi^\T\vx_i}\psi_{\sigma}(\vxi)
	\end{equation*}
	has finite Sobolev norm $\norm{\fF^{-1}[\phi_\sigma]}_{H^{\frac{\alpha}{2}}}<\infty$, where $\psi_{\sigma}(\vxi)$ satisfies $\fF^{-1}[\psi_{\sigma}](\vx)=\E^{-\frac{\norm{\vx}^2}{2\sigma^2}}$, $\mA=\left(\exp(-\frac{\norm{\vx_j-\vx_i}^2}{2\sigma^2})\right)_{n\times n}$ and  $\vg^{(\sigma)}=\left(g^{(\sigma)}_1,g^{(\sigma)}_2,\cdots,g^{(\sigma)}_n\right)^\T=\mA^{-1}\vY$. Thus $m<+\infty$.

	
	
	\noindent 3.
	Choose a minimizing sequence $\{\bar{\phi}_k\}_{k=1}^\infty\subset \overline{\fA_{\vX,\vY}}$ such that
	\begin{equation*}
        \lim_{k\rightarrow \infty} \norm{\fF^{-1}[\bar{\phi}_k]}_{H^{\frac{\alpha}{2}}} =m.
	\end{equation*}
	By definition of the closure, there exists a function $\phi_k\in\fA_{\vX,\vY}$ for each $k$ such that
	\begin{equation*}
	    \norm{\fF^{-1}[\bar{\phi}_k]-\fF^{-1}[\phi_k]}_{H^{\frac{\alpha}{2}}}\leq \frac{1}{k}.
	\end{equation*}
	Therefore $\{\phi_k\}_{k=1}^\infty\subset \fA_{\vX,\vY}$ is also a minimizing sequence, i.e.,
	\begin{equation*}
        \lim_{k\rightarrow \infty} \norm{\fF^{-1}[\phi_k]}_{H^{\frac{\alpha}{2}}} =m.
	\end{equation*}	
	Then $\{\fF^{-1}[\phi_k]\}_{k=1}^\infty$ is bounded in the Sobolev space $H^{\frac{\alpha}{2}}(\sR^d)$. Hence there exist a weakly convergent subsequence $\{\fF^{-1}[\phi_{n_k}]\}_{k=1}^\infty$ and a function $\fF^{-1}[\phi^*]\in H^{\frac{\alpha}{2}}(\sR^d)$ such that
	\begin{equation*} 
	\fF^{-1}[\phi_{n_k}]\rightharpoonup \fF^{-1}[\phi^*] \quad\text{in } H^{\frac{\alpha}{2}}(\sR^d)\text{ as}\  k\rightarrow \infty.
	\end{equation*}
	Note that
	\begin{equation*} 
    m=\inf_{\phi\in \overline{\fA_{\vX,\vY}}}\norm{\fF^{-1}[\phi]}_{H^{\frac{\alpha}{2}}}\le 
	\norm{\fF^{-1}[\phi^*]}_{H^{\frac{\alpha}{2}}}\le\liminf_{\phi_{n_k}}\norm{\fF^{-1}[\phi_{n_k}]}_{H^{\frac{\alpha}{2}}}=m,
	\end{equation*}
	where we have used the lower semi-continuity of the Sobolev norm of $H^{\frac{\alpha}{2}}(\sR^d)$ in the third inequality. Hence $\norm{\fF^{-1}[\phi^*]}_{H^{\frac{\alpha}{2}}}=m$.
	
    \noindent 4. We further establish the strong convergence that
    $\fF^{-1}[\phi_{n_k}]-\fF^{-1}[\phi^*]\rightarrow 0$ in $H^{\frac{\alpha}{2}}(\sR^d)$ as $k\rightarrow \infty$.
	In fact, since $\fF^{-1}[\phi_{n_k}]\rightharpoonup \fF^{-1}[\phi^*] \ \text{in } H^{\frac{\alpha}{2}}(\sR^d)\text{ as}\ k\rightarrow \infty$ and $\lim_{k\rightarrow \infty}\norm{\fF^{-1}[\phi_{n_k}]}_{H^{\frac{\alpha}{2}}}=m=\norm{\fF^{-1}[\phi^*]}_{H^{\frac{\alpha}{2}}}$, we have
     \begin{align*}
     	&\lim_{k\to \infty}\norm{\fF^{-1}[\phi_{n_k}]-\fF^{-1}[\phi^*]}_{H^{\frac{\alpha}{2}}}^2=\lim_{k\to \infty} \langle\fF^{-1}[\phi_{n_k}]-\fF^{-1}[\phi^*],\fF^{-1}[\phi_{n_k}]-\fF^{-1}[\phi^*]\rangle\\
     	&=\lim_{k\to \infty} \langle\fF^{-1}[\phi_{n_k}],\fF^{-1}[\phi_{n_k}]\rangle+ \langle\fF^{-1}[\phi^*],\fF^{-1}[\phi^*]\rangle- \langle\fF^{-1}[\phi_{n_k}],\fF^{-1}[\phi^*]\rangle- \langle\fF^{-1}[\phi^*],\fF^{-1}[\phi_{n_k}]\rangle\\
     	&=m^2+m^2- \lim_{k\to \infty}\left(\langle\fF^{-1}[\phi_{n_k}],\fF^{-1}[\phi^*]\rangle+ \langle\fF^{-1}[\phi^*],\fF^{-1}[\phi_{n_k}]\rangle\right)\\
     	&=m^2+m^2-\langle\fF^{-1}[\phi^*],\fF^{-1}[\phi^*]\rangle-\langle\fF^{-1}[\phi^*],\fF^{-1}[\phi^*]\rangle=0.
     \end{align*}
     Here $\langle\cdot,\cdot \rangle$ is the inner product of the Hilbert space $H^{\frac{\alpha}{2}}$.
     
	\noindent 5. 
	We have $\phi^*\in L^1(\sR^d)$ because
	\begin{align*}
		\int_{\sR^d}\abs{\phi^*(\vxi)}\diff \vxi =\int_{\sR^d}\frac{\langle\vxi\rangle^{\frac{\alpha}{2}}\abs{\phi^*(\vxi)}}{\langle\vxi\rangle^{\frac{\alpha}{2}}}\diff \vxi
		\le\norm{\fF^{-1}[\phi^*]}_{H^{\frac{\alpha}{2}}} \left(\int_{\sR^d}\frac{1}{\langle\vxi\rangle^\alpha}\diff\vxi\right)^{\frac{1}{2}}=Cm<+\infty,
	\end{align*}
	where $C:=\left(\int_{\sR^d}\frac{1}{\langle\vxi\rangle^\alpha}\diff\vxi\right)^{\frac{1}{2}}<+\infty$. Hence $\phi^*\in L^1(\sR^d)\cap L^2(\sR^d)$ and $\fP_{\vX}\phi^*$ is well-defined.

	\noindent 6. Recall that $\fP_{\vX}\phi_{n_k}=\vY$. We have
	\begin{align*}
		\Abs{\vY-\fP_{\vX}\phi^*}
		&=\lim_{k\to+\infty}\Abs{\fP_{\vX}\phi_{n_k}-\fP_{\vX}\phi^*}\\
		&=
		\lim_{k\to+\infty}\Abs{\int_{\sR^d}(\phi_{n_k}-\phi^*)\E^{2\pi\I\vx\vxi}\diff\vxi}\\
		&=\lim_{k\to+\infty}\Abs{\int_{\sR^d}\frac{\langle\vxi\rangle^{\frac{\alpha}{2}}(\phi_{n_k}-\phi^*)}{\langle\vxi\rangle^{\frac{\alpha}{2}}}\E^{2\pi\I\vx\vxi} \diff\vxi}\\
		&\le\lim_{k\to+\infty}\norm{\fF^{-1}[\phi_{n_k}]-\fF^{-1}[\phi^*]}_{H^{\frac{\alpha}{2}}} \left(\int_{\sR^d}\frac{\Abs{\E^{2\pi\I\vx\vxi}}^2}{\langle\vxi\rangle^\alpha}\diff\vxi\right)^{\frac{1}{2}}\\
		&=C\lim_{k\to+\infty}\norm{\fF^{-1}[\phi_{n_k}]-\fF^{-1}[\phi^*]}_{H^{\frac{\alpha}{2}}}=0.
	\end{align*}
	Hence $\fP_{\vX}\phi^*=\vY$ and $\phi^*\in \fA_{\vX,\vY}$.
	
	\noindent 7. 
	Note that
	\begin{equation*}
	    m=\inf_{\phi\in \overline{\fA_{\vX,\vY}}}\norm{\fF^{-1}[\phi]}_{H^{\frac{\alpha}{2}}}
	    \leq
	    \inf_{\phi\in \fA_{\vX,\vY}}\norm{\fF^{-1}[\phi]}_{H^{\frac{\alpha}{2}}}
	    \leq\norm{\fF^{-1}[\phi^*]}_{H^{\frac{\alpha}{2}}}
	    =m.
	\end{equation*}
	This implies that $\inf_{\phi\in \fA_{\vX,\vY}}\norm{\fF^{-1}[\phi]}_{H^{\frac{\alpha}{2}}}=m$ and
	$
        \phi^*\in\arg\min_{\phi\in \fA_{\vX,\vY}}\norm{\fF^{-1}[\phi]}_{H^{\frac{\alpha}{2}}}
    $, which completes the proof.
\end{proof}

\end{document}